\documentclass[12pt]{amsart}

\usepackage[utf8]{inputenc}
\usepackage{amsmath}
\usepackage{amssymb}
\usepackage{amsfonts}
\usepackage{amsthm}
\usepackage{thmtools}
\usepackage{enumerate}
\usepackage[top=3cm, bottom=3cm, left=3cm, right=3cm]{geometry}
\usepackage{graphicx}
\usepackage[all]{xy}
\usepackage{titlesec}
\usepackage{hyperref}
\usepackage{rotating}
\usepackage{xcolor}
\usepackage{setspace}
\usepackage[T1]{fontenc}
\usepackage{verbatim}

\titleformat{\section}
{\filcenter\large\bf} 
{\thesection.\ } 
{1pt} 
{} %

\titleformat{\subsection}[hang]
{\filcenter\bf}
{\thesubsection.}
{1pt}
{}

\declaretheoremstyle[bodyfont=\normalfont]{normalbody}

\declaretheorem[numberwithin=section,name=Theorem]{theorem}

\declaretheorem[sibling=theorem,name=Lemma]{lemma}
\declaretheorem[sibling=theorem,name=Proposition]{proposition}

\declaretheorem[sibling=theorem,style=normalbody,name=Remark]{remark}

\newcommand{\Z}{\mathbb{Z}}

\newcommand{\Q}{\mathbb{Q}}

\newcommand{\G}{\mathbb{G}}
\newcommand{\A}{\mathbb{A}}

\newcommand{\Hom}{\operatorname{Hom}}

\newcommand{\Gal}{\operatorname{Gal}}

\newcommand{\ab}{\operatorname{ab}}

\newcommand{\Res}{\operatorname{Res}}
\newcommand{\Inf}{\operatorname{Inf}}
\newcommand{\Cores}{\operatorname{Cores}}

\newcommand{\im}{\operatorname{im}}
\newcommand{\id}{\operatorname{id}}

\newcommand{\Pic}{\mathrm{Pic}}

\DeclareFontFamily{U}{wncy}{}
\DeclareFontShape{U}{wncy}{m}{n}{<->wncyr10}{}
\DeclareSymbolFont{mcy}{U}{wncy}{m}{n}
\DeclareMathSymbol{\Sh}{\mathord}{mcy}{"58}

\DeclareFontFamily{U}{wncy}{}
\DeclareFontShape{U}{wncy}{m}{n}{<->wncyr10}{}
\DeclareSymbolFont{mcy}{U}{wncy}{m}{n}
\DeclareMathSymbol{\Ch}{\mathord}{mcy}{"51}

\title{Hasse norm principle for Galois dihedral extensions}

\author{Felipe Rivera-Mesas}
\address{Departamento de Matemáticas, Facultad de Ciencias, Universidad de Chile}
\email{felipe.rivera.m@ug.uchile.cl}

\begin{document}

\maketitle

\begin{abstract}
	Let $L/k$ an Galois extension of number fields with Galois group isomorphic to a dihedral group of order $2n$. In this note, we give a general description of the Hasse norm principle for $L/k$ and the weak approximation for the norm one torus $R^1_{L/k}(\G_m)$ associated to $L/k$.
\end{abstract}

%
%
%
%

\section{Introduction}

Given a number field extension $L/k$, we say that the Hasse norm principle (HNP) holds for $L/k$ if every element of $L$ which is a local norm everywhere is a global norm. In other words, if $\A_L$ denotes the adèle ring and $N_{L/k}$ denotes the adelic norm induced by the norm $N_{L/k}:L^*\to k^*$, then the HNP holds for $L/k$ if the group
	\[ \frak{K}(L/k)=\frac{k^*\cap N_{L/k}(\A_L^*)}{N_{L/k}(L^*)} \]
is trivial. Furthermore, we can get a geometric interpretation of the Hasse norm principle for $L/k$ as follows: we have the following exact sequence of algebraic tori
	\[ 1 \to R^1_{L/k}(\G_m) \to R_{L/k}(\G_m) \overset{N_{L/k}}{\to} \G_{m,k} \to 1, \]
where $R_{L/k}(\G_m)$ denotes the Weil restriction of $\G_m$ from $L$ to $k$ and $N_{L/k}$ is the map induced by the norm from $L$ to $k$. The torus $R^1_{L/k}(\G_m)$ is called \emph{the norm one torus associated to} $L/k$. In this way, the group $\frak{K}(L/k)$ is isomorphic to the Tate-Shafarevic group $\Sh^1(k,R^1_{L/k}(\G_m))$ of $R^1_{L/k}(\G_m)$. Now, recalling that $\Sh^1(k,R^1_{L/k}(\G_m))$ is the obstruction to Hasse Principle for torsors under $R^1_{L/k}(\G_m)$, we have that the HNP holds for $L/k$ if and only if the Hasse Principle holds for torsors under $R^1_{L/k}(\G_m)$.

Given a $k$-torus $T$, we say that Weak Approximation (WA) holds for $T$ if its $k$-points are dense in the product of its local points. Equivalently, Weak Approximation holds for $T$ if the group
	\[ A(T) = \left(\prod_v T(k_v)\right)/\overline{T(k)} \]
is trivial.

Voskresenski\u\i\ in \cite{Vosk70} gives the following exact sequence, which makes a link between the Hasse principle for torsors under a torus $T$ and the WA for $T$:
	\[ 0 \to A(T) \to H^1(k,\Pic\overline{X})^\sim \to \Sh^1(k,T) \to 0, \]
where $\overline{X}$ is a smooth compatification of $T$ and $(-)^\sim=\Hom(-,\Q/\Z)$. In particular, when $T$ is the norm one torus $R^1_{L/k}(\G_m)$ associated to a number field extension $L/k$, the Voskresenski\u\i\ exact sequence links the HNP for $L/k$ and the WA for $T$. When the extension $L/k$ is Galois we have the following two facts:
	\begin{itemize}
		\item Colliot-Thélène has proved \cite[Proposition 7]{CT77} that the central object of the Voskresenski exact sequence is isomorphic to $H^3(\Gal(L/k),\Z)$;
		\item Tate has given \cite[p. 198]{Cassels2010} an explicit description of $\Sh^1(k,T)$ as follows:
			\[ \Sh^1(k,T)^\sim \cong \ker \left[ \Res : H^3(\Gal(L/k),\Z) \to \prod_{v\in\Omega_k} H^3(G_v,\Z) \right], \]
		where $\Omega_k$ denotes the set of all places of $k$ and $G_v$ denotes the descomposition group of $v$. Thus, we also have the following isomorphism:
			\[ A(T)^\sim \cong \im \left[ \Res : H^3(\Gal(L/k),\Z) \to \prod_{v\in\Omega_k} H^3(G_v,\Z) \right]. \]
	\end{itemize}

We will focus on this case, i.e. when $L/k$ is Galois. In this notes, we will give a general description of the HNP for $L/k$ and the WA for $T$ when $L/k$ has Galois group isomorphic to $D_n$, the dihedral group of order $2n$.

The following result gives us a general criterion for determining when a Galois extension of number fields with Galois group a dihedral group satisfies the HNP or when its norm one torus associated $R^1_{L/k}(\G_m)$ has WA.

\begin{theorem} \label{main}
	Let $L/k$ be a Galois extension of number fields with Galois group isomorphic to $D_n$, the dihedral group of order $2n$. Let $T:=R^1_{L/k}(\G_m)$ the norm one torus associated to $L/k$. Then, we have the following 
		\begin{enumerate}
			\item if $n$ is odd, then the HNP holds for $L/k$ and the WA holds for $T$;
			\item if $n$ is even, then
				\begin{enumerate}
					\item If there exists $v\in\Omega_k$ such that $G_v$ contains a Klein subgroup of $G$, then HNP holds for $L/k$ and WA fails for $T$;
					\item If for every $v\in\Omega_k$ the group $G_v$ does not contain a Klein subgroup, then HNP fails for $L/k$ and WA holds for $T$.
				\end{enumerate}
		\end{enumerate}
	In particular, the WA holds for $T$ if and only if for every $v\in\Omega_k$ the group $G_v$ does not contain a Klein subgroup of $G$.
\end{theorem}

\section{HNP for $L/k$ and WA for $T:=R^1_{L/k}(\G_m)$}

The goal of this section is to give an explicit description of the group $H^3(G,\Z)$ when $G$ is a dihedral group and a general description of the restriction maps
	\[ H^3(G,\Z) \to H^3(H,\Z), \]
for every $H\leq G$. Let $G:=D_n$ the dihedral group of orden $2n$, i.e.
	\[ G = \langle r,s \mid r^n=s^2=e,\ srs=r^{-1} \rangle. \]

\subsection{Description of the group $H^3(G,\Z)$}

The following proposition gives an explicit description of $H^3(G,\Z)$.

\begin{proposition} \label{H3dihedral}
	Let $G$ be a dihedral group of order 2n. The group $H^3(G,\Z)$ is isomorphic to $(\Z/n\Z)[2]$.
\end{proposition}
\begin{proof}
	From the Hochschild-Serre spectral sequence associated to the split exact sequence
		\[ 1 \to \Z/n\Z \to G \to \Z/2\Z \to 1 \]
	and the constant $G$-module $\Q/\Z$, we have the following exact sequence:
		\[ H^2(\Z/2\Z,\Q/\Z) \to \ker(\Res) \to H^1(\Z/2\Z,H^1(\Z/n\Z,\Q/\Z)) \overset{\lambda}{\to} H^3(\Z/2\Z,\Q/\Z), \]
	where $\Res:=\Res: H^2(G,\Q/\Z) \to H^2(\Z/n\Z,\Q/\Z)$. Since $\Z/2\Z$ is cyclic and $\Q$ is uniquely divisible, we have
		\[ H^2(\Z/2\Z,\Q/\Z) \cong H^1(\Z/2\Z,\Z) = 0, \]
	and similarly we have $H^2(\Z/n\Z,\Q/\Z)=0$. Thus, the previous exact sequence becomes
		\[ 0 \to H^2(G,\Q/\Z) \to H^1(\Z/2\Z,H^1(\Z/n\Z,\Q/\Z)) \overset{\lambda}{\to} H^3(\Z/2\Z,\Q/\Z). \]
	Now, the exact sequence $1\to\Z/n\Z\to G\to\Z/2\Z\to 1$ is split. Then, since $\Q/\Z$ is constant, we have that the map
		\[ \Inf : H^3(\Z/2\Z,\Q/\Z) \to H^3(G,\Q/\Z), \]
	is injective. On the other hand, from the Hochschild-Serre spectral sequence, we have the complex
		\[ H^1(\Z/2\Z,H^1(\Z/n\Z,\Z)) \overset{\lambda}{\to} H^3(\Z/2\Z,\Q/\Z) \overset{\Inf}{\to} H^3(G,\Q/\Z). \]
	Hence, $\lambda$ is trivial and then $H^2(G,\Q/\Z)\cong H^1(\Z/2\Z,H^1(\Z/n\Z,\Q/\Z))$. Now,
		\[ H^1(\Z/2\Z,H^1(\Z/n\Z,\Q/\Z))\cong H^1(\Z/2\Z,\Z/n\Z), \] 
	with $\Z/2\Z$ acting on $\Z/n\Z$ by taking the inverse. Thus we have that
		\[ H^3(G,\Z) \cong H^2(G,\Q/\Z) \cong H^1(\Z/2\Z,\Z/n\Z) \cong (\Z/n\Z)[2], \]
	where the last isomorphism is obtained from a direct computation on cocycles and coboundaries.
\end{proof}

\subsection{Study of restriction maps $H^3(G,\Z)\to H^3(H,\Z)$.}

Now, in order to find a general description for $\Sh^1(k,T)$, we will study all restriction maps 
	\[ \Res_{G/H} : H^3(G,\Z) \to H^3(H,\Z), \]
with $H\leq G$.

For convenience, we will introduce the following notation: for a subgroup $H\leq G$ and a $G$-module $M$, we denote by $\Res_{G/H}^i(M)$ the restriction map
	\[ \Res : H^i(G,M) \to H^i(H,M). \]

The following lemma helps us for this purpose.

\begin{lemma} \label{index2}
	Suppose that $n$ is even. If $H\leq G$ is a dihedral subgroup of index 2 and order divisible by 4, then $\ker(\Res_{G/H}^3(\Z))$ is trivial.
\end{lemma}

\begin{remark} \label{klein}
	Note that the condition about the divisibility by 4 of the order of $H$ implies that $H$ contains a Klein subgroup of $G$.
\end{remark}

\begin{proof}
	Without loss of generality, we may suppose that $H=\langle r^2,s\rangle$. We have the following exact sequence
		\[ 0 \to H \to G \to \Z/2\Z \to 0, \]
	where $\Z/2\Z$ is generated by $\overline{r}$. Now, by the Hochschild-Serre spectral sequence, we have the inclusion
		\[ \ker(\Res^2_{G/H}(\Q/\Z)) \hookrightarrow H^1(\Z/2\Z,H^1(H,\Q/\Z)) \cong H^1(\Z/2\Z,H^{\ab}), \]
	where $\Z/2\Z$ acts on $H^{\ab}$ via ${^{\overline{r}}}(h[H,H])=(rhr^{-1})[H,H]$. Since $H$ contains a Klein subgroup of $G$ (cf. Remark \ref{klein}), we have that
		\[ H^{\ab} = \langle s[H,H],sr^2[H,H] \rangle \cong \Z/2\Z\times\Z/2\Z. \]
	Now, since $\Z/2\Z$ is cyclic, we have that
		\[ H^1(\Z/2\Z,H^{\ab}) \cong H^{-1}(\Z/2\Z,H^{\ab}) = \ker(N)/\langle 1-\overline{r}\rangle H^{\ab}, \]
	where $N:H^{\ab}\to H^{\ab}$ is the norm map. We note that
		\[ N(s[H,H])=N(sr^2[H,H])=r^2[H,H]\neq [H,H] \]
	and 
		\[ (1-\overline{r})s[H,H] = (1-\overline{r})(sr^2[H,H]) = r^2[H,H]. \]
	Hence, $H^1(\Z/2\Z,H^{\ab})$ is trivial. Therefore, since $H^i(G,\Z)\cong H^{i-1}(G,\Q/\Z)$ for $i>1$, $\ker(\Res^3_{G/H}(\Z))\cong\ker(\Res_{G/H}^2(\Q/\Z))$ is trivial.
		
\end{proof}

The following result gives a general description of the restriction maps $\Res_{G/H}^3(\Z)$. 

\begin{proposition} \label{res}
	Let $G\cong D_n$ and $H\leq G$. Then
		\[ \Res : H^3(G,\Z) \to H^3(H,\Z) = \begin{cases} \id & \text{if }H\text{ contains a Klein subgroup of }G; \\
																0 & \text{if }H\text{ does not contain a Klein subgroup of }G.
													\end{cases} \]
\end{proposition}
\begin{proof}
	Let us recall that all subgroups of $G$ are either cyclic or a dihedral group $D_k$ with $k$ dividing $n$. Now, If $H$ is cyclic or a dihedral group $D_k$ with $k$ odd, then $H^3(H,\Z)=0$ by Proposition \ref{H3dihedral}. Thus, if $H$ has order non divisible by 4, then $\Res^3_{G/H}(\Z)$ is trivial. In particular, when $n$ is odd, all restriction maps $\Res_{G/H}^3$ are trivial. Then, we only have to prove the proposition when $n$ is even and $H$ is a dihedral group $D_k$ with $k$ even. If $n=2^{v_2(n)}\ell$, let $H'\cong D_m$ with $m=2^{v_2(k)}\ell$ and $H\leq H'$. Now, we can get a sequence of groups
		\[ H'_0=G \supseteq H_1' \supseteq ... \supseteq H'_r=H', \]
	where $r=v_2(n)-v_2(k)$, $H'_i\cong D_{k_i}$ and $[H'_{i-1}:H_i']=2$ for all $i>0$. Then, by Lemma \ref{index2}, we have that
		\[ \Res : H^3(H'_{i-1},\Z) \to H^3(H'_i,\Z)) \]
	is injective for all $i>0$.
	On the other hand, by Proposition \ref{H3dihedral}, we have that $\Res:H^3(H',\Z)\to H^3(H,\Z)$ is injective since $[H':H]$ is odd and $\Cores\circ\Res=[H':H]$. Hence,
		\[ \Res_{G/H} = \Res_{G/H'_1}\circ ... \circ \Res_{H'_{r-1}/H'} \circ \Res_{H'/H} \]
	is injective.
\end{proof}

The results above allows us to prove the Theorem \ref{main}:

\begin{proof}[Proof of the Theorem \ref{main}]
	When $n$ is odd the central object of the Voskresenski\u\i\ exact sequence is trivial and therefore $\Sh^1(k,T)=A(T)=0$. On the other hand, when $n$ is even, the group $H^3(G,\Z)$ is isomorphic to $\Z/2\Z$, and then the Voskresenski\u\i\ exact sequence becomes
		\[ 0 \to A(T) \to \Z/2\Z \to \Sh^1(k,T) \to 0, \]
	Besides, we note that $\Sh^1(k,T)$ is trivial iff there exists $v\in\Omega_k$ such that the restriction map
		\[ \Res_v : H^3(G,\Z) \to H^3(G_v,\Z) \] 
	is not trivial and, by Proposition \ref{res}, this occurs iff there exists $v\in\Omega_k$ such that $G_v$ contains a Klein group. Whence we conclude the proof of the theorem.

\end{proof}

\bibliographystyle{alpha}
\bibliography{dihedral}

\begin{thebibliography}{CTS77}

\bibitem[CF10]{Cassels2010}
J.W.S. Cassels and A.~Fr{\"o}hlich.
\newblock {\em Algebraic Number Theory: Proceedings of an Instructional
  Conference Organized by the London Mathematical Society (a NATO Advanced
  Study Institute) with the Support of the International Mathematical Union}.
\newblock London Mathematical Society, 2010.

\bibitem[CTS77]{CT77}
Jean-Louis Colliot-Th\'{e}l\`ene and Jean-Jacques Sansuc.
\newblock La {$R$}-\'{e}quivalence sur les tores.
\newblock {\em Ann. Sci. \'{E}cole Norm. Sup. (4)}, 10(2):175--229, 1977.

\bibitem[Vos70]{Vosk70}
V~E Voskresenski{\u{\i}}.
\newblock Birational properties of linear algebraic groups.
\newblock {\em Mathematics of the {USSR}-Izvestiya}, 4(1):1--17, feb 1970.

\end{thebibliography}

\end{document}